\documentclass[a4paper,10pt]{article}

\title{On A New Class of Tempered Stable Distributions: \\ Moments and Regular Variation}
\author{Michael Grabchak\footnote{Email address: mgrabcha@uncc.edu}, \textit{University of North Carolina Charlotte}}

\usepackage[T1]{fontenc}
\usepackage{amsfonts}
\usepackage{verbatim}
\usepackage{amsmath, amsthm, amssymb}

\normalsize
\begin{document}

\newtheorem{prop}{Proposition}
\newtheorem{thrm}[prop]{Theorem}
\newtheorem{defn}[prop]{Definition}
\newtheorem{cor}[prop]{Corollary}

\newtheorem{proposition}{Proposition}
\newtheorem{theorem}[prop]{Theorem}
\newtheorem{definition}[prop]{Definition}
\newtheorem{corollary}[prop]{Corollary}

\newtheorem{lemma}[prop]{Lemma}
\newtheorem{remark}[prop]{Remark}
\newcommand{\rd}{\mathrm d}
\newcommand{\rE}{\mathrm E}
\newcommand{\ts}{TS^p_\alpha}
\newcommand{\ets}{ETS^p_\alpha}
\newcommand{\tr}{\mathrm{tr}}
\newcommand{\iid}{\stackrel{\mathrm{iid}}{\sim}}
\newcommand{\eqd}{\stackrel{d}{=}}
\newcommand{\cond}{\stackrel{d}{\rightarrow}}
\newcommand{\conv}{\stackrel{v}{\rightarrow}}
\newcommand{\conw}{\stackrel{w}{\rightarrow}}
\newcommand{\conp}{\stackrel{p}{\rightarrow}}
\newcommand{\confdd}{\stackrel{fdd}{\rightarrow}}
\newcommand{\plim}{\mathop{\mathrm{p\mbox{-}lim}}}

\maketitle

\begin{abstract}
 
We extend the class of tempered stable distributions first introduced in Rosi\'nski 2007 \cite{Rosinski:2007}. Our new class allows for more structure and more variety of tail behaviors. We discuss various subclasses and the relation between them. To characterize the possible tails we give detailed results about finiteness of various moments. We also give necessary and sufficient conditions for the tails to be regularly varying. This last part allows us to characterize the domain of attraction to which a particular tempered stable distribution belongs.

\end{abstract}

\section{Introduction}

Tempered stable distributions were defined in Rosi\'nski 2007 \cite{Rosinski:2007} as a class of models obtained by modifying the L\'evy measures of stable distributions by multiplying their densities by completely monotone functions. This allows for models that are similar to stable distributions in some central region, but possess lighter (i.e.\ tempered) tails. It has been observed that these models provide a good fit to data in a variety of applications. These include mathematical finance \cite{Carr:Geman:Madan:Yor:2002} \cite{Kim:Rachev:Bianchi:Fabozzi:2010:Tempered}, biostatistics \cite{Aalen:1992} \cite{Palmer:Ridout:Morgan:2008}, computer science \cite{Terdik:Gyires:2009b}, and physics \cite{Bruno:Sorriso-Valvo:Carbone:Bavassano:2004} \cite{Meerschaert:Zhang:Baeumer:2008}. An explanation for why such models might appear in applications is given in \cite{Grabchak:Samorodnitsky:2010}.

The purpose of this paper is two-fold. First, we provide necessary and sufficient conditions for tempered stable distributions to have regularly varying tails. This is an important question both from a theoretical perspective, since it will allow us to classify which domain of attraction a tempered stable distribution belongs to, and from an applied point of view, since such models are often used in practice. Our second purpose is to introduce the class of $p$-tempered $\alpha$-stable distributions, where $p>0$ and $\alpha<2$. The parameter $p$ controls the amount of tempering, while $\alpha$ is the index of stability of the corresponding stable distribution. Clearly the case where $\alpha\le0$ no longer has any meaning in terms of tempering stable distributions, however it allows the class to be more flexible. In fact, within certain subclasses, the case where $\alpha\le 0$ has been shown to provide a good fit to data, see e.g.\ \cite{Aalen:1992} or \cite{Carr:Geman:Madan:Yor:2002}.

This class combines a number of important subclasses that have been studied separately in the literature. In particular, when $p=1$ and $\alpha\in(0,2)$ it coincides with Rosi\'nski's \cite{Rosinski:2007} tempered stable distributions. When $p=2$ and $\alpha\in[0,2)$ it coincides with the class of tempered infinitely divisible distributions defined in \cite{Bianchi:Rachev:Kim:Fabozzi:2011}. If we allow the distributions to have a Gaussian part, then we would have the class $J_{\alpha,p}$ defined in \cite{Maejima:Nakahara:2009}. This, in turn, contains important subclasses including the Thorin class (when $p=1$ and $\alpha=0$), the Goldie-Steutel-Bondesson class (when $p=1$ and $\alpha=-1$), the class of type $M$ distributions (when $p=2$ and $\alpha=0$), and the class of type $G$ distributions (when $p=2$ and $\alpha=-1$). For more information on these classes see \cite{Barndorff-Nielsen:Maejima:Sato:2006}, \cite{Aoyama:Maejima:Rosinski:2008}, and the references therein.

This paper is structured as follows. In Section \ref{sec: defns} we define $p$-tempered $\alpha$-stable distributions and state some basic results. We show that, as with tempered stable distributions, for a fixed $\alpha$ and $p$, all elements of this class are uniquely determined by a Rosi\'nski measure $R$ and a shift $b$. The remaining two sections are concerned with relating the tails of the Rosi\'nski measure to the tails of the distribution. In Section \ref{sec: moments} we give necessary and sufficient conditions for the existence of moments and exponential moments. We also give explicit formulas for the cumulants. Finally, in Section \ref{sec: Reg Var for TS} we give necessary and sufficient conditions for the tails to be regularly varying. Specifically, we show that the tails of a $p$-tempered $\alpha$-stable distribution are regularly varying if and only if the tails of the corresponding Rosi\'nski measure are regularly varying.

Before proceeding, recall that the characteristic function of an infinitely divisible distribution $\mu$ on $\mathbb R^d$  can be written as $\hat\mu(z) = \exp\{C_{\mu}(z)\}$ where
\begin{eqnarray}\label{eq: inf div char func}
C_{\mu}(z) = -\frac{1}{2}\langle z,Az\rangle + i\langle b,z\rangle + \int_{\mathbb R^d}\left(e^{i\langle z,x\rangle}-1-i\frac{\langle z,x\rangle}{1+|x|^2}\right)M(\rd x),
\end{eqnarray}
$A$ is a symmetric nonnegative-definite $d\times d$ matrix, $b\in\mathbb R^d$, and $M$ satisfies
\begin{eqnarray}\label{eq: cond for levy measure}
M(\{0\})=0 \mbox{\ and\ } \int_{\mathbb R^d}(|x|^2\wedge1)M(\rd x)<\infty.
\end{eqnarray}
The measure $\mu$ is uniquely identified by the L\'evy triplet $(A,M,b)$ and we write $\mu= ID(A,M,b)$.

\section{$p$-Tempered $\alpha$-Stable Distributions}\label{sec: defns}

Recall that for $\alpha\in(0,2)$ the L\'evy measure of an $\alpha$-stable distribution with spectral measure $\sigma$ is given by
\begin{eqnarray}\label{eq: M alpha in TS}
L(B) = \int_{\mathbb S^{d-1}}\int_0^\infty 1_B(ru)r^{-\alpha-1}\rd r\sigma(\rd u), \quad B\in\mathfrak B(\mathbb R^d).
\end{eqnarray}
By analogy, we define the following.

\begin{definition} \label{defn: pTalphaS}
Fix $\alpha<2$ and $p>0$. An infinitely divisible probability measure $\mu$ is called a $\mathbf p$\textbf{-tempered} $\mathbf \alpha$\textbf{-stable distribution} if it has no Gaussian part and its L\'evy measure is given by
\begin{eqnarray}\label{eq:m}
M(B) =\int_{\mathbb S^{d-1}}\int_0^\infty 1_B(ru)q(r^p,u)r^{-\alpha-1}\rd r\sigma(\rd u), \quad B\in\mathfrak B(\mathbb R^d),
\end{eqnarray}
where $\sigma$ is a finite Borel measure on $\mathbb S^{d-1}$ and $q:(0,\infty)\times\mathbb S^{d-1}\mapsto\mathbb (0,\infty)$ is a Borel function such that for all $u\in\mathbb S^{d-1}$ $q(\cdot,u)$ is completely monotone and
\begin{eqnarray}\label{eq: q goes to zero}
\lim_{r\rightarrow\infty}q(r,u)=0.
\end{eqnarray}
We denote the class of $p$-tempered $\alpha$-stable distributions by $\ts$. If, in addition, $\lim_{r\downarrow0}q(r,u)=1$ for every $u\in \mathbb S^{d-1}$ then $\mu$ is called a \textbf{proper} $\mathbf p$\textbf{-tempered} $\mathbf \alpha$\textbf{-stable distribution}.
\end{definition}

\begin{remark}\label{remark: Bern}
The complete monotonicity of $q(\cdot,u)$ implies that for each $u\in\mathbb S^{d-1}$ the function $q(r,u)$ is differentiable and monotonely decreasing in $r$. Moreover, by Bernstein's Theorem (see e.g.\ Theorem 1a in Section XIII.4 of \cite{Feller:1971}),
\begin{eqnarray}\label{eq:Q(|)}
q(r^p,u) = \int_{(0,\infty)} e^{-r^ps}Q_u(\mathrm d s)
\end{eqnarray}
for some measurable family $\{Q_u\}_{u\in \mathbb S^{d-1}}$ of Borel measures on $(0,\infty)$. For a guarantee that we can take the family to be measurable see Remark 3.2 in \cite{Barndorff-Nielsen:Maejima:Sato:2006}. Note that the condition $\lim_{r\downarrow0}q(r,u)=1$ for every $u\in \mathbb S^{d-1}$ is equivalent to the condition that $\left\{Q_u\right\}_{u\in \mathbb S^{d-1}}$ is a family of probability measures.
\end{remark}

\begin{remark}
From \eqref{eq:Q(|)} it follows that as $p$ increases the tails of $M$ (as given in \eqref{eq:m}) go to zero
quicker. In this sense $p$ controls the extent to which the tails of the L\'evy measure are tempered.
\end{remark}

\begin{remark}
For $\alpha\in(0,2)$ and $p>0$, all proper $p$-tempered $\alpha$-stable distributions belong to the class of generalized tempered stable distributions defined in \cite{Rosinski:Sinclair:2010}. Many important results about their L\'evy processes are given there. These include short time behavior, conditions for absolute continuity with respect to the underlying stable process, and a series representation, see Theorems 3.1, 4.1, and 5.5 in \cite{Rosinski:Sinclair:2010} for details. 
\end{remark}

\begin{remark}
From Theorem 15.10 in \cite{Sato:1999} it follows that $p$-tempered $\alpha$-stable distributions are selfdecomposable if and only if $q(r^p,u)r^{-\alpha}$ is a decreasing function of $r$ for every $u\in\mathbb S^{d-1}$. By Remark \ref{remark: Bern} this always holds when $\alpha\in[0,2)$. Thus, when $\alpha\in[0,2)$, $p$-tempered $\alpha$-stable distributions inherit properties of selfdecomposable distributions. In particular, if they are nondegenerate then they are absolutely continuous with respect to Lebesgue measure in $d$-dimensions and when $d=1$ they are unimodal.
\end{remark}

Following \cite{Rosinski:2007}, we will reparametrize the L\'evy measure $M$ into a form that is often easier to work with. Let $Q$ be a Borel measure on $\mathbb{R}^d$ given by
\begin{eqnarray}\label{eq: Q}
Q(A) = \int_{\mathbb S^{d-1}}\int_{(0,\infty)} 1_A(ru)Q_u(\rd r)\sigma(\rd u), &
A\in\mathfrak{B}(\mathbb{R}^d).
\end{eqnarray}
Note that $Q(\{0\})=0$. Define a Borel measure $R$ on $\mathbb{R}^d$ by
\begin{eqnarray}\label{eq:R}
R(A) =
\int_{\mathbb{R}^d}1_A\left(\frac{x}{|x|^{1+1/p}}\right)|x|^{\alpha/p}Q(\rd x), & A\in\mathfrak{B}(\mathbb{R}^d),
\end{eqnarray}
and again note that $R(\{0\})=0$. To get the inverse transformation we have
\begin{eqnarray}\label{eq: Q in terms of R}
Q(A) = \int_{\mathbb{R}^d}1_A\left(\frac{x}{|x|^{p+1}}\right)|x|^{\alpha}R(\rd x), & A\in\mathfrak{B}(\mathbb{R}^d).
\end{eqnarray}
The following result extends Theorem 2.3 of \cite{Rosinski:2007}.

\begin{theorem}\label{thrm:cond on R}
Fix $p>0$. Let $M$ be given by \eqref{eq:m} and let $R$ be given by \eqref{eq:R}.\\
1.  We can write
\begin{eqnarray}\label{eq:levy m}
M(A) = \int_{\mathbb{R}^d}\int_0^\infty
1_A(tx)t^{-1-\alpha}e^{-t^p}\rd t R(\rd x), &
A\in\mathfrak{B}(\mathbb{R}^d),
\end{eqnarray}
or equivalently,
\begin{eqnarray}\label{eq:levy m2}
M(A) = p^{-1}\int_{\mathbb{R}^d}\int_0^\infty
1_A(t^{1/p}x)t^{-1-\alpha/p}e^{-t}\rd t R(\rd x), &
A\in\mathfrak{B}(\mathbb{R}^d).
\end{eqnarray}
2. \eqref{eq:levy m} defines a L\'evy measure if and only if either $R=0$ or the following hold: 
\begin{eqnarray}\label{eq: R(0)=0}
\alpha<2,\ R(\{0\})=0,
\end{eqnarray}
and
\begin{eqnarray}\label{eq: integ cond on R}
\int_{\mathbb R^d} \left(|x|^2\wedge|x|^\alpha\right)R(\rd x)<\infty & \mathrm{if} \ \alpha\in(0,2),\nonumber\\
\int_{\mathbb R^d} \left(|x|^2\wedge[1+\log^+|x|]\right)R(\rd x)<\infty & \mathrm{if}  \ \alpha=0,\\
\int_{\mathbb R^d} \left(|x|^2\wedge1\right)R(\rd x)<\infty & \mathrm{if} \ \alpha<0.\nonumber
\end{eqnarray}
Moreover, when $R$ satisfies these conditions, $M$ is the L\'evy measure of a $p$-tempered $\alpha$-stable distribution and it uniquely determines $R$.\\
3.  A $p$-tempered $\alpha$-stable distribution is proper if and only if in addition to \eqref{eq: R(0)=0} and \eqref{eq: integ cond on R} $R$ satisfies
\begin{eqnarray}\label{eq: cond for proper}
\int_{\mathbb R^d}|x|^\alpha R(\rd x)<\infty.
\end{eqnarray}
4. If $R$ satisfies \eqref{eq: cond for proper} then in \eqref{eq:m} the measure $\sigma$ is given by
\begin{eqnarray}
\sigma(B) = \int_{\mathbb R^d} 1_B\left(\frac{x}{|x|}\right)|x|^\alpha R(\rd x), \quad B\in\mathfrak B(\mathbb S^{d-1}).
\end{eqnarray}
\end{theorem}

Note that for all $\alpha<2$ the conditions in \eqref{eq: integ cond on R} imply the necessity of
\begin{eqnarray}
\int_{\mathbb R^d}\left(|x|^2\wedge|x|^\alpha\right)R(\rd x)<\infty\ \mbox{and}\ \int_{\mathbb R^d}\left(|x|^2\wedge1\right)R(\rd x)<\infty.
\end{eqnarray}
Before proving Theorem \ref{thrm:cond on R}, we will translate the integrability conditions on $R$ into integrability conditions on $\{Q_u\}_{u\in\mathbb S^{d-1}}$ and $\sigma$.

\begin{corollary}\label{corr:cond on Q(|)}
Fix $p>0$, let $M$ be given by \eqref{eq:m}, and let $\{Q_u\}$ be as in \eqref{eq:Q(|)}. $M$ is a L\'evy measure if and only if either
$$
Q_u(\mathbb R_+) = 0 \ \sigma\mbox{-a.e.}
$$
or $\alpha<2$ and
\begin{eqnarray}
\int_{\mathbb S^{d-1}}\int_0^\infty\left(t^{-(2-\alpha)/p}\wedge1\right)Q_u(\rd t)\sigma(\rd u)<\infty & \alpha\in(0,2)\nonumber\\
\int_{\mathbb S^{d-1}}\int_0^\infty\left(t^{-2/p}\wedge\left[1+\log^+(t^{-1/p})\right]\right)Q_u(\rd t)\sigma(\rd u)<\infty & \alpha=0\nonumber\\
\int_{\mathbb S^{d-1}}\int_0^\infty\left(t^{-(2-\alpha)/p} \wedge t^{\alpha/p} \right)Q_u(\rd t)\sigma(\rd u)<\infty & \alpha<0\nonumber.
\end{eqnarray}
\end{corollary}

Note that these conditions guarantee that for any $p>0$ and for $\sigma$-a.e. $u$ \eqref{eq: q goes to zero} holds and $\int_{\mathbb R^d} e^{-r^ps}Q_u(\rd s)<\infty$.

\begin{proof}[Proof of Theorem \ref{thrm:cond on R}]
We omit most parts of the proof because they are similar to the case when $p=1$ and $\alpha\in(0,2)$, which is given in \cite{Rosinski:2007}. We only show that when $M$ is given by \eqref{eq:levy m} it is a L\'evy measure if and only if \eqref{eq: R(0)=0} and \eqref{eq: integ cond on R} hold. Assume $R\ne0$, since the other case is trivial. We have
$$
M(\{0\}) =\int_{\mathbb{R}^d}\int_0^\infty1_{\{0\}}(tx)t^{-\alpha-1}e^{-t^p}\rd tR(\rd x) = \int_{\{0\}}\int_0^\infty t^{-1-\alpha}e^{-t^p}\rd tR(\rd x),
$$
which equals zero if and only if $R(\{0\})=0$.

Now assume $\int(|x|^2\wedge1)M(\rd x)<\infty$. For any $\epsilon>0$
\begin{eqnarray}
\infty &>& \int_{|x|\le1}|x|^2M(\rd x) =
\int_{\mathbb R^d}|x|^2\int_0^{|x|^{-1}}t^{1-\alpha}e^{-t^p}\rd tR(\rd x)\nonumber\\
&\ge& \int_{|x|\le1/\epsilon}|x|^2\int_0^\epsilon t^{1-\alpha}e^{-t^p}\rd tR(\rd x) \ge
e^{-\epsilon^p}\int_{|x|\le1/\epsilon}|x|^2 \int_0^\epsilon t^{1-\alpha}\rd tR(\rd x)\nonumber.
\end{eqnarray}
Since $R\ne0$, for this to be finite for all $\epsilon>0$ it is necessary that $\alpha<2$. Taking $\epsilon=1$ gives the necessity of $\int_{|x|\le1}|x|^2R(\rd x)<\infty$. Observe that
\begin{eqnarray}
\infty &>& \int_{|x|\ge1}M(\rd x) = \int_{\mathbb{R}^d}\int_{|x|^{-1}}^\infty t^{-1-\alpha}e^{-t^p}\rd tR(\rd x)\nonumber\\
&\ge& \int_{|x|\ge1}\int_{|x|^{-1}}^\infty t^{-1-\alpha}e^{-t^p}\rd tR(\rd x)\nonumber\\
&\ge& \int_1^\infty t^{-1-\alpha}e^{-t^p}\rd t\int_{|x|\ge1}R(\rd x) + e^{-1}\int_{|x|\ge1}\int_{|x|^{-1}}^1 t^{-1-\alpha}\rd tR(\rd x).\nonumber
\end{eqnarray}
This implies the necessity of $\int_{|x|\ge1}R(\rd x)<\infty$ and $\int_{|x|\ge1}\int_{|x|^{-1}}^1
t^{-1-\alpha}\rd t R(\rd x)<\infty$. When $\alpha<0$ we are done. When $\alpha=0$ this implies the finiteness
of $\int_{|x|\ge1}\log|x|R(\rd x)$, and when $\alpha\in(0,2)$ it implies the finiteness of
$\int_{|x|\ge1}|x|^\alpha R(\rd x)$.
Thus \eqref{eq: R(0)=0} and \eqref {eq: integ cond on R} hold.

Now assume that \eqref{eq: R(0)=0} and \eqref{eq: integ cond on R} hold. We have
\begin{eqnarray*}
&& \int_{|x|\le1}|x|^2M(\rd x) = \int_{\mathbb{R}^d}|x|^2\int_0^{|x|^{-1}}t^{1-\alpha}e^{-t^p}\rd t R(\rd x)\\
&&\quad \le \int_{|x|\le1}|x|^2R(\rd x)\int_0^\infty t^{1-\alpha}e^{-t^p}\rd t + \int_{|x|>1}|x|^2\int_0^{|x|^{-1}}
t^{1-\alpha}\rd tR(\rd x)\\
&&\quad = p^{-1}\Gamma\left(\frac{2-\alpha}{p}\right)\int_{|x|\le1}|x|^2R(\rd x) +(2-\alpha)^{-1}\int_{|x|>1}|x|^{\alpha}R(\rd x)<\infty.
\end{eqnarray*}
Let $D=\sup_{t\ge1}t^{2-\alpha}e^{-t^p}$. We have
\begin{eqnarray*}
&&\int_{|x|\ge1}M(\mathrm \rd x) = \int_{\mathbb{R}^d}\int_{|x|^{-1}}^\infty
t^{-1-\alpha}e^{-t^p}\rd t R(\rd x)\\
&&\quad \le D\int_{|x|\le1}\int_{|x|^{-1}}^\infty t^{-3}\rd t R(\rd x) +
\int_{|x|>1}\int_{|x|^{-1}}^\infty t^{-1-\alpha}e^{-t^p}\rd t R(\rd x).
\end{eqnarray*}
The first integral in the above equals
$.5D\int_{|x|\le1}|x|^2R(\rd x)$, which is assumed finite. The
second integral can be written as
\begin{eqnarray*}
\int_{|x|>1}\int_{|x|^{-1}}^1 t^{-1-\alpha}e^{-t^p}\rd t R(\rd x) +
\int_{1}^\infty t^{-1-\alpha}e^{-t^p}\rd t\int_{|x|>1}R(\rd x)
\end{eqnarray*}
Of these, the second integral is finite since $\int_{|x|>1}R(\rd x)<\infty$. The first is bounded by
$\int_{|x|>1}\frac{|x|^\alpha-1}{\alpha}R(\rd x)$ when $\alpha\ne0$ and by $\int_{|x|>1}\log|x|R(\rd x)$, when
$\alpha=0$. The fact that both of these are assumed to be finite gives the result.
\end{proof}

\begin{definition}
The unique measure in \eqref{eq:R} is called the \textbf{Rosi\'nski measure} of the corresponding $p$-tempered $\alpha$-stable distribution.
\end{definition}

\begin{remark}
For $\alpha\in(0,2)$ and $p=1$ the Rosi\'nski measure was called the spectral measure in
\cite{Rosinski:2007}. For $\alpha\in[0,2)$ and $p=2$ the Rosi\'nski measure was introduced in a slightly
different parametrization in \cite{Bianchi:Rachev:Kim:Fabozzi:2011}.
\end{remark}

\begin{remark}
Fix $\alpha<2$, $p>0$, and let $\mu\in \ts$ with Rosi\'nski measure $R$. Then $\mu=ID(0,M,b)$ for some $b\in\mathbb R^d$ and $M$ uniquely determined by $R$. We write $\ts(R,b)$ to denote this distribution.
\end{remark}

Theorem \ref{thrm:cond on R} shows that for a fixed $p>0$ and $\alpha<2$, the Rosi\'nski measure is uniquely determined by the L\'evy measure. This leaves the question of whether all of the parameters are jointly identifiable. Unfortunately this is not the case. As we will show below, even for a fixed $p>0$ the parameters $\alpha$ and $R$ are not jointly identifiable. However, using ideas similar to those in \cite{Rosinski:2007}, we will show that for a fixed $p>0$, in the subclass of proper tempered stable distribution, they are jointly identifiable. On the other hand, for a fixed $\alpha<2$, even in the subclass of proper tempered stable distributions, the parameters $p$ and $R$ are not jointly identifiable. We begin with the following lemma.

\begin{lemma}\label{lemma: for ident}
Fix $\alpha<2$, $p>0$, and let $M$ be the L\'evy measure of a $p$-tempered $\alpha$-stable distribution with Rosi\'nski measure $R\ne0$.\\
1. The map $s\mapsto s^\alpha M(|x|>s)$ is decreasing and $\lim_{s\rightarrow\infty}s^\alpha M(|x|>s)=0$.\\
2. If $\alpha\in(0,2)$ then
$$
\lim_{s\downarrow 0}s^{\alpha}M(|x|>s) = \frac{1}{\alpha}\int_{\mathbb{R}^d}|x|^\alpha R(\mathrm \rd x)
$$
and if $\alpha\le0$ then
$$
\lim_{s\downarrow0}s^{\alpha}M(|x|>s) = \infty.
$$
3. If $\alpha<0$ then
$$
\lim_{s\downarrow0}s^{\alpha}M(|x|<s) = \frac{1}{|\alpha|}\int_{\mathbb{R}^d}|x|^\alpha R(\rd x)
$$
and if $\alpha\in[0,2)$ then for all $s>0$
$$
M(|x|<s) = \infty.
$$
\end{lemma}

This lemma extends Corollary 2.5 in \cite{Rosinski:2007}. Note that it implies that in the subclass of proper tempered stable distributions both $\lim_{s\downarrow0}s^{\alpha}M(|x|>s)=\infty$ and $M(|x|<s) = \infty$ if and only if $\alpha=0$.

\begin{proof}
We begin with the first part. Since
\begin{eqnarray}
s^\alpha M(|x|>s) &=& s^\alpha\int_{\mathbb R^d}\int_{s|x|^{-1}}^\infty t^{-1-\alpha}e^{-t^p}\rd tR(\rd x)\nonumber\\
&=& \int_{\mathbb{R}^d}\int_{|x|^{-1}}^\infty t^{-1-\alpha}e^{-(st)^p}\rd t R(\rd x),\label{eq: for levy
tails}
\end{eqnarray}
the map $s\mapsto s^\alpha M(|x|>s)$ is decreasing. For large enough $s$, the integrand in \eqref{eq: for levy tails} is bounded by $t^{-1-\alpha}e^{-t^p}$, which is integrable. Thus by dominated convergence $\lim_{s\rightarrow\infty}s^\alpha M(|x|>s)=0$.

For the second part, by \eqref{eq: for levy tails} and the Monotone Convergence Theorem
\begin{eqnarray*}
\lim_{s\downarrow0}s^{\alpha}M(|x|>s) &=& \int_{\mathbb R^d}\int_{|x|^{-1}}^\infty t^{-1-\alpha}\rd t R(\rd x).
\end{eqnarray*}
Thus if $\alpha\in(0,2)$ then
$$
\lim_{s\downarrow0}s^{\alpha}M(|x|>s) = \frac{1}{\alpha}\int_{\mathbb R^d}|x|^\alpha R(\rd x),
$$
and if $\alpha\le0$ then
$$
\lim_{s\downarrow0}s^{\alpha}M(|x|>s) = \infty.
$$

Now for the third part. If $\alpha\in[0,2)$ then for all $s>0$
\begin{eqnarray*}
M(|x|<s) &=& \int_{\mathbb R^d}\int_0^{s|x|^{-1}}t^{-1-\alpha}e^{-t^p}\rd tR(\rd x)\\
&\ge& \int_{\mathbb R^d}e^{-(s/|x|)^p}\int_0^{s|x|^{-1}}t^{-1-\alpha}\rd tR(\rd x) = \infty.
\end{eqnarray*}
If $\alpha<0$ then
\begin{eqnarray*}
\lim_{s\downarrow0}s^\alpha M(|x|<s) &=& \lim_{s\downarrow0}s^\alpha \int_{\mathbb
R^d}\int_0^{s|x|^{-1}}t^{-1-\alpha}e^{-t^p}\rd tR(\rd x)\\
&=& \lim_{s\downarrow0}\int_{\mathbb R^d}\int_0^{|x|^{-1}} t^{-1-\alpha}e^{-(st)^p}\rd t R(\rd x)\\
&=& \int_{\mathbb{R}^d}\int_0^{|x|^{-1}} t^{-1-\alpha}\rd t R(\rd x)
=\frac{1}{|\alpha|}\int_{\mathbb{R}^d}|x|^\alpha R(\rd x),
\end{eqnarray*}
where the third line follows by the Monotone Convergence Theorem.
\end{proof}

Combining Lemma \ref{lemma: for ident} with \eqref{eq: cond for proper} gives the following.

\begin{proposition}
In the subclass of proper tempered stable distributions with parameter $p>0$ fixed, the parameters $R$ and $\alpha$ are jointly identifiable.
\end{proposition}

However, in general, the parameters $\alpha$ and $p$ are not identifiable. This will become apparent from the
following results.

\begin{proposition}\label{prop: change alphas}
Fix $\alpha<2$, $\beta\in(\alpha,2)$, and let $K=\int_0^\infty
s^{\beta-\alpha-1}e^{-s^p}\rd s$. If $\mu=TS^p_\beta(R,b)$ and
\begin{eqnarray}\label{eq: R for change alphas}
R'(A) = K^{-1}\int_{\mathbb R^d}\int_0^1 1_A(ux) u^{-\beta-1}\left(1-u^p\right)^{(\beta-\alpha)/p-1}\rd uR(\rd x)
\end{eqnarray}
then $R'$ is the Rosi\'nski measure of a $p$-tempered $\alpha$-stable distribution and
$\mu=TS^p_\alpha(R',b)$.
\end{proposition}

\begin{proof}
First we will show that $R'$ is the Rosi\'nski measure of some $p$-tempered $\alpha$-stable distribution.  Let $C=\max_{u\in[0,.5]}\left(1-u^p\right)^{(\beta-\alpha)/p-1}$. We have
\begin{eqnarray*}
K\int_{|x|\le 1} |x|^2 R'(\rd x) &=& \int_{\mathbb R^d}|x|^2
\int_0^{1\wedge|x|^{-1}}u^{1-\beta}(1-u^p)^{(\beta-\alpha)/p-1}\rd u R(\rd x)\\
&\le& \int_{|x|\le2}|x|^2 R(\rd x)\int_0^1 u^{1-\beta}(1-u^p)^{(\beta-\alpha)/p-1}\rd u\\
&& \ \  + C \int_{|x|>2}|x|^2 \int_0^{|x|^{-1}} u^{1-\beta} \rd u R(\rd x)\\
&=& \int_{|x|\le2}|x|^2 R(\rd x)\int_0^1 u^{1-\beta}(1-u^p)^{(\beta-\alpha)/p-1}\rd u\\
&& \ \  + \frac{C}{2-\beta}\int_{|x|\ge2}|x|^\beta R(\rd x)<\infty.
\end{eqnarray*}
If $\alpha\in(0,2)$ then
\begin{eqnarray*}
K\int_{|x|>2} |x|^\alpha R'(\rd x) &=& \int_{|x|\ge2}|x|^\alpha \int_{|x|^{-1}}^{1/2} u^{\alpha-1-\beta}(1-u^p)^{(\beta-\alpha)/p-1}\rd u R(\rd x) \\
&& \ \ + \int_{|x|\ge2}|x|^\alpha \int_{1/2}^1 u^{\alpha-1-\beta}(1-u^p)^{(\beta-\alpha)/p-1}\rd u R(\rd x) \\
&\le& C \int_{|x|\ge2}|x|^\alpha \int_{|x|^{-1}}^\infty u^{\alpha-1-\beta}\rd u R(\rd x) \\
&& \ \ + \int_{|x|>2}|x|^\beta  R(\rd x)\int_{1/2}^1 u^{\alpha-1-\beta}(1-u^p)^{(\beta-\alpha)/p-1}\rd u.
\end{eqnarray*}
Here the first integral equals $\frac{C}{\beta-\alpha} \int_{|x|\ge2}|x|^\beta R(\rd x)<\infty$ and the second is also finite. Now assume $\alpha=0$ and fix $\epsilon\in(0,\beta)$. By 4.1.37 in \cite{Abramowitz:Stegun:1972} there exists a $C_\epsilon>0$ such that for all $u>0$, $\log u\le C_\epsilon u^\epsilon$. Thus
\begin{eqnarray*}
K\int_{|x|>2} \log|x| R'(\rd x) &\le&  KC_\epsilon \int_{|x|>2} |x|^\epsilon R'(\rd x),
\end{eqnarray*}
which is finite by arguments similar to the previous case. When $\alpha<0$
\begin{eqnarray*}
K\int_{|x|>2}R'(\rd x) &=& \int_{|x|>2}\int^1_{|x|^{-1}}u^{-1-\beta}(1-u^p)^{(\beta-\alpha)/p-1}\rd uR(\rd x)\\
&\le& C\int_{|x|>2}\int^{1}_{|x|^{-1}}u^{-1-\beta}\rd uR(\rd x)\\
&& \ \ +\int_{|x|>2}R(\rd x)\int^1_{1/2}u^{-1-\beta}(1-u^p)^{(\beta-\alpha)/p-1}\rd u.
\end{eqnarray*}
Here the second integral is finite. For $\beta\ne0$ the first equals $\frac{C}{\beta} \int_{|x|>2}\left(|x|^\beta-1\right)R(\rd x)$ which is finite, and for $\beta=0$ it equals
$\int_{|x|>2}\log|x|R(\rd x)<\infty$.

Now, let $M'$ be the L\'evy measure of $\ts(R',b)$. By \eqref{eq:levy m}
\begin{eqnarray*}
&& M'(A) =  K^{-1}\int_{\mathbb R^d}\int_0^\infty\int_0^1 1_A(utx)t^{-1-\alpha}e^{-t^p} u^{-\beta-1}\left(1-u^p\right)^{\frac{\beta-\alpha}{p}-1}\rd u\rd tR(\rd x)\\
&& \qquad = K^{-1}\int_{\mathbb R^d}\int_0^\infty\int_0^t 1_A(vx)t^{\beta-\alpha-1}e^{-t^p}
v^{-\beta-1}\left(1-v^p/t^p\right)^{\frac{\beta-\alpha}{p}-1}\rd v\rd tR(\rd x)\\
&& \qquad = K^{-1}\int_{\mathbb R^d}\int_0^\infty\int_v^\infty 1_A(vx)t^{p-1}e^{-t^p}
v^{-\beta-1}\left(t^p-v^p\right)^{\frac{\beta-\alpha}{p}-1}\rd t\rd vR(\rd x)\\
&& \qquad =  K^{-1}\int_{\mathbb R^d}\int_0^\infty 1_A(vx)e^{-v^p}v^{-\beta-1}\rd v R(\rd x) \int_0^\infty
e^{-s^p} s^{\beta-\alpha-1}\rd s\\
&& \qquad = \int_{\mathbb R^d}\int_0^\infty 1_A(vx)e^{-v^p}v^{-\beta-1}\rd vR(\rd x),
\end{eqnarray*}
where the second line follows by the substitution $v=ut$ and the fourth by the substitution $s^p = t^p-v^p$.
\end{proof}

To show a similar result for the parameter $p$ we need some additional notation. For $r\in(0,1)$, let $f_r$ be the density of the
$r$-stable distribution with
\begin{eqnarray}
\int_0^\infty e^{-tx} f_r(x)\rd x = e^{-t^r}.
\end{eqnarray}
Such a density exists by Proposition 1.2.12 in \cite{Samorodnitsky:Taqqu:1994}. However, the only case where an
explicit formula is known is
\begin{eqnarray}
f_{.5}(s) = (2\sqrt\pi)^{-1}e^{-1/(4s)}s^{-3/2}1_{[s>0]}
\end{eqnarray}
(see Examples 2.13 and 8.11 in \cite{Sato:1999}). From Theorem 5.4.1 in \cite{Uchaikin:Zolotarev:1999} it follows that if $X\sim f_r$ and $\beta
\ge 0$ then
\begin{eqnarray}\label{eq: neg moments of stable}
E|X|^{-\beta}<\infty.
\end{eqnarray}

\begin{proposition}\label{prop: change ps}
Fix $\alpha<2$, $0<p<q$. If $\mu = TS^p_\alpha(R,b)$ and
\begin{eqnarray}
R'(A) = \int_{\mathbb R^d}\int_0^\infty 1_A(s^{-1/q}x) s^{\alpha/q}f_{p/q}(s)\rd s R(\rd x)
\end{eqnarray}
then $R'$ is the Rosi\'nski measure of a $q$-tempered $\alpha$-stable distribution and $\mu = TS^q_\alpha(R',b)$. Moreover, $\mu$ is a proper $p$-tempered $\alpha$-stable distribution if and only if it is a proper $q$-tempered $\alpha$-stable distribution.
\end{proposition}

This implies that for a fixed $\alpha$ the parameters $p$ and $R$ are not jointly identifiable even within the subclass of proper tempered stable distributions.

\begin{proof}
First we show that $R'$ is, in fact, the Rosi\'nski measure of a $q$-tempered $\alpha$-stable distribution. We have
\begin{eqnarray*}
\int_{|x|\le 1}|x|^2 R'(\rd x) &=& \int_{\mathbb R^d}|x|^2 \int_{|x|^q}^\infty s^{-(2-\alpha)/q} f_{p/q}(s) \rd s R(\rd x)\\
&\le& \int_{|x|\le1}|x|^2 \int_0^\infty s^{-(2-\alpha)/q} f_{p/q}(s) \rd s R(\rd x)\\
&& \quad +\int_{|x|>1}|x|^\alpha R(\rd x) \int_0^\infty f_{p/q}(s) \rd s <\infty.
\end{eqnarray*}
If $\alpha\ne0$ and $\beta=\alpha\vee0$ then
\begin{eqnarray*}
\int_{|x|>1}|x|^\beta R'(\rd x) &=& \int_{\mathbb R^d}|x|^\beta \int_0^{|x|^q} s^{-(\beta-\alpha)/q} f_{p/q}(s) \rd s R(\rd x)\\
&\le& \int_{|x|\le1}|x|^2 \int_0^\infty s^{-(2-\alpha)/q} f_{p/q}(s) \rd s R(\rd x)\\
&& \quad +\int_{|x|>1}|x|^\beta \int_0^\infty s^{-(\beta-\alpha)/q} f_{p/q}(s) \rd s R(\rd x)<\infty.
\end{eqnarray*}
If $\alpha=0$ then
\begin{eqnarray*}
\int_{|x|>1}\log|x|R'(\rd x) &=& \int_{\mathbb R^d}\int_0^{|x|^q}\log|x s^{-1/q}| f_{p/q}(s) \rd s R(\rd x)\\
&\le& \int_{|x|\le1}|x|^2 R(\rd x)\int_0^\infty s^{-2/q} f_{p/q}(s) \rd s\\
&& \quad \int_{|x|>1}\log|x| R(\rd x)\int_0^\infty f_{p/q}(s) \rd s \\
&& \quad \int_{|x|>1} R(\rd x) \int_0^\infty s^{-1/q} f_{p/q}(s) \rd s<\infty,
\end{eqnarray*}
where the inequality uses the fact that $\log |x|\le |x|$ (see 4.1.36 in \cite{Abramowitz:Stegun:1972}).

If $M'$ is the L\'evy measure of $TS^q_\alpha(R',b)$ then by \eqref{eq:levy m}
\begin{eqnarray*}
M'(A) &=& \int_{\mathbb R^d}\int_0^\infty\int_0^\infty 1_A(s^{-1/q}tx)t^{-1-\alpha}e^{-t^q}\rd t s^{\alpha/q}f_{p/q}(s)\rd s R(\rd x) \\
&=& \int_{\mathbb R^d}\int_0^\infty 1_A(vx)v^{-1-\alpha}\int_0^\infty e^{-v^qs}f_{p/q}(s)\rd s \rd v R(\rd x) \\
&=& \int_{\mathbb R^d}\int_0^\infty 1_A(vx)v^{-1-\alpha}e^{-v^p} \rd v R(\rd x).
\end{eqnarray*}

The last part follows from \eqref{eq: cond for proper} and the fact that
\begin{eqnarray*}
\int_{\mathbb R^d}|x|^\alpha R'(\rd x) = \int_{\mathbb R^d}|x|^\alpha R(\rd x) \int_0^\infty s^{-\alpha/q} s^{\alpha/q} f_{p/q}(s) \rd s = \int_{\mathbb R^d}|x|^\alpha R(\rd x).
\end{eqnarray*}
This completes the proof.
\end{proof}

Propositions \ref{prop: change alphas} and \ref{prop: change ps} give a constructive proof of the following
result, a version of which was shown in \cite{Maejima:Nakahara:2009}.

\begin{corollary} \label{prop: not ident} Fix $\alpha<2$, $p>0$, and let $\mu\in TS^p_\alpha$.\\
1. For any $q\ge p$, $\mu\in TS^q_\alpha$.\\
2. For any $\beta\le\alpha$, $\mu\in TS^p_{\beta}$.
\end{corollary}

We end this section by characterizing when a $p$-tempered $\alpha$-stable distribution is $\beta$-stable
for some $\beta\in(0,2)$.

\begin{proposition}
Fix $\alpha<2$, $p>0$, and $\beta\in(0,2)$. Let $\mu$ be a $\beta$-stable distribution with spectral measure $\sigma\ne0$. If $\beta\le\alpha$ then $\mu\notin\ts$. If $\beta\in(0\vee\alpha,2)$ then $\mu= TS^p_\alpha(R,b)$ and
\begin{eqnarray}\label{eq: R meas of beta stable}
R(A) = K^{-1}\int_{\mathbb S^{d-1}}\int_0^\infty 1_A(r\xi) r^{-1-\beta}\rd r \sigma(\rd\xi), \quad A\in\mathfrak B(\mathbb R^d),
\end{eqnarray}
where $K = \int_0^\infty t^{\beta-\alpha-1}e^{-t^p}\rd t$.
\end{proposition}

Note that
$$
\int_{\mathbb R^d}|x|^\alpha R(\rd x)= K^{-1}\sigma(\mathbb S^{d-1})\int_0^\infty r^{-(\beta-\alpha)-1}\rd r =\infty.
$$
Thus, by Part 3 of Theorem \ref{thrm:cond on R}, no stable distributions are proper $p$-tempered
$\alpha$-stable.

\begin{proof}
If $\mu\in\ts$ then its L\'evy measure can be written as \eqref{eq:m}. By uniqueness of the polar
decomposition of L\'evy measures (see Lemma 2.1 in \cite{Barndorff-Nielsen:Maejima:Sato:2006}) the function
$q(r,u)=r^{(\alpha-\beta)/p}$. This is not completely monotone when $\beta<\alpha$, and it does not
satisfy \eqref{eq: q goes to zero} when $\beta=\alpha$.

Now assume that $\beta>\alpha$ and let $R$ be as in \eqref{eq: R meas of beta stable}. In this case $R(\{0\})=0$ and for any $\gamma\in[0,\beta)$
$$
\int_{\mathbb R^d}\left(|x|^2\wedge|x|^\gamma\right)R(\rd x) = K^{-1}\sigma(\mathbb S^{d-1})\int_0^\infty (r^{1-\beta}\wedge r^{\gamma-\beta-1})\rd r<\infty.
$$
Thus, by Theorem \ref{thrm:cond on R}, $R$ is the Rosi\'nski measure of a $p$-tempered $\alpha$-stable distribution. If $M$ is the L\'evy measure of $\ts(R,b)$ then
\begin{eqnarray*}
M(A) &=& K^{-1}\int_{\mathbb S^{d-1}}\int_0^\infty\int_0^\infty 1_A(rt\xi) t^{-1-\alpha}e^{-t^p}\rd t r^{-1-\beta}\rd r\sigma(\rd\xi)\\
&=& K^{-1}\int_0^\infty t^{\beta-\alpha-1} e^{-t^p} \rd t \int_{\mathbb S^{d-1}} \int_0^\infty  1_A(r\xi) r^{-1-\beta} \rd r\sigma(\rd\xi)\\
&=& \int_{\mathbb S^{d-1}} \int_0^\infty  1_A(r\xi) r^{-1-\beta} \rd r\sigma(\rd\xi),
\end{eqnarray*}
which is the L\'evy measure of $\mu$.
\end{proof}

\section{Moments}\label{sec: moments}

In this section we give necessary and sufficient conditions for the finiteness of moments and exponential moments. We also give explicit formulas for the cumulants when they exist. This is useful, for instance, in parameter estimation by the method of moments.  First we introduce some notation. For any $x\in\mathbb R^d$ let $x_i$ be the $i$th component. For simplicity, throughout this section, we will use $M$ to denote the L\'evy measure of a $p$-tempered $\alpha$-stable distribution.

Let $k$ be a $d$-dimensional vector of nonnegative integers. Let $C_\mu$ be as in \eqref{eq: inf div char func}. Recall that we define the cumulant
\begin{eqnarray}
c_k = (-i)^{\sum k_i} \frac{\partial^{\sum k_i}}{\partial z_d^{k_d}\cdots\partial z_1^{k_1}}C_\mu(z) \Big|_{z=0},
\end{eqnarray}
when the derivative exists and is continuous in a neighborhood of zero. Cumulants can be uniquely expressed in
terms of moments. Let $X\sim \mu$. When $k_i=1$ and $k_j=0$ for all $j\ne i$ then $c_k=\rE X_i$,  when $k_i=2$
and $k_j=0$ for all $j\ne i$ then $c_k=\mathrm{var}(X_i)$, and when for some $i\ne j$ $k_i=k_j=1$ and
$k_\ell=0$ for all $\ell\ne i,j$ then $c_k = \mathrm{cov}(X_i,X_j)$. In the statement of the following
theorem, we adopt the convention that $0^0=1$.

\begin{theorem}\label{thrm:moments}
Fix $\alpha<2$, $p>0$ and let $\mu=TS_\alpha^p(R,b)$.\\
1. If $\alpha\in(0,2)$ and $q_1,\dots,q_d\ge0$ with $q:=\sum_{j=1}^d q_j < \alpha$ then
\begin{eqnarray}\label{eq: always finite for less than alpha}
\int_{\mathbb{R}^d}\left(\prod_{j=1}^d|x_j|^{q_j}\right)\mu(\rd x) \le \int_{\mathbb{R}^d}|x|^q\mu(\rd
x)<\infty.
\end{eqnarray}
2. If $\alpha\in(0,2)$ then
\begin{eqnarray}\label{eq: when norm finite at alpha}
\int_{\mathbb{R}^d}|x|^\alpha\mu(\rd x)<\infty \Longleftrightarrow \int_{|x|>1}|x|^\alpha\log|x| R(\rd x)<\infty.
\end{eqnarray}
Additionally, if $q_1,\dots,q_d\ge0$ with $\sum_{j=1}^d q_j = \alpha$ then
\begin{eqnarray}
\int_{\mathbb{R}^d}\left(\prod_{j=1}^d|x_j|^{q_j}\right)\mu(\rd x)<\infty
\end{eqnarray}
if and only if
\begin{eqnarray}\label{eq: conf for finite moment equal alpha}
\int_{|x|>1}\left(\prod_{j=1}^d|x_j|^{q_j}\right)\log|x| R(\rd x)<\infty.
\end{eqnarray}
3. If $q>\left(\alpha\vee0\right)$ then
\begin{eqnarray}\label{eq: conf for finite moment bigger than alpha}
\int_{\mathbb{R}^d}|x|^q\mu(\rd x)<\infty \Longleftrightarrow \int_{|x|>1}|x|^qR(\rd x)<\infty.
\end{eqnarray}
Additionally, if $q_1,\dots,q_d\ge0$ with $\sum_{j=1}^d q_j>\left(\alpha\vee0\right)$ then
\begin{eqnarray}\label{eq: mixed moments finite mu}
\int_{\mathbb{R}^d}\left(\prod_{j=1}^d|x_j|^{r_j}\right)\mu(\rd x)<\infty \mbox{\ for\ all\ } r_k\in[0,q_k],
\
k=1,\dots,d
\end{eqnarray}
if and only if
\begin{eqnarray}\label{eq: mixed moments finite R}
\int_{|x|>1}\left(\prod_{j=1}^d|x_j|^{r_j}\right) R(\rd x)<\infty \mbox{\ for\ all\ } r_k\in[0,q_k], \
k=1,\dots,d.
\end{eqnarray}
4. Let $q_1,\dots,q_d$ be nonnegative integers and let $q=\sum_{i=1}^d q_i$. Further if $q=\alpha$
assume that \eqref{eq: conf for finite moment equal alpha} holds and if $q>\alpha$ that \eqref{eq: mixed
moments finite R} holds. If $q_i = q=1$ for some $i$ then
\begin{eqnarray}
c_{(q_1,\dots,q_d)}= b_i + \int_{\mathbb R^d}\int_0^\infty x_i \frac{|x|^2}{1+|x|^2t^2} t^{2-\alpha} e^{-t}\rd t R(\rd x).
\end{eqnarray}
If $q\ge 2$ then
\begin{eqnarray}
c_{(q_1,\dots,q_d)} = p^{-1}\Gamma\left(\frac{n-\alpha}{p}\right)\int_{\mathbb R^d}\left(\prod_{j=1}^d x_j^{q_j}\right) R(\rd x).
\end{eqnarray}
\end{theorem}

For proper $1$-tempered $\alpha$-stable distributions with $\alpha\in(0,2)$ a somewhat weaker version of Part
(iv) above was given in \cite{Terdik:Woyczynski:2006}.

\begin{proof}
By Corollary 25.8 in \cite{Sato:1999}, the condition $\int_{\mathbb{R}^d}|x|^q\mu(\rd x)<\infty$ is equivalent
to the condition $\int_{|x|>1}|x|^qM(\rd x)<\infty$. Similarly, by Theorem 1 in \cite{Sapatinas:Shanbhag:2010}
the condition $\int_{\mathbb{R}^d}\left(\prod_{j=1}^d|x_j|^{r_j}\right)\mu(\rd x)<\infty$ for all
$r_k\in[0,q_k]$, $k=1,\dots,d$ is equivalent to the condition
$\int_{|x|>1}\left(\prod_{j=1}^d|x_j|^{r_j}\right)M(\rd x)<\infty$ for all $r_k\in[0,q_k]$, $k=1,\dots,d$.

We will now transfer the integrability conditions from $M$ to $R$. Let $f_q(x)$ be either $|x|^q$ or
$\prod_{j=1}^d|x_j|^{r_j}$ where $\sum^d_{j=1} r_j=q$. By \eqref{eq:levy m}
\begin{eqnarray*}
\int_{|x|>1}f_q(x)M(\rd x) = \int_{\mathbb R^d}\int_{|x|^{-1}}^\infty f_q(x) t^{q-1-\alpha}e^{-t^p}\rd t R(\rd x).
\end{eqnarray*}
From here  \eqref{eq: when norm finite at alpha} and Parts 1 and 2 follow by arguments
similar to those in Proposition 2.7 of \cite{Rosinski:2007}. The second half of Part 2 essentially
follows from arguments similar to those in Proposition 2.7 of \cite{Rosinski:2007} as well, but to guarantee
that the integral remains finite for all $r_k\in[0,q_k)$ we use \eqref{eq: always finite for less than
alpha}.

For general infinitely divisible distributions, the form of the cumulants in terms of the L\'evy measure is
given in Theorems 5.1 and 5.2 of
\cite{Gupta:Shanbhag:Nguyen:Chen:2009}. From this, Part 4 follows by using \eqref{eq:levy m} and
simplifying.
\end{proof}

In the rest of this section we will give conditions for the finiteness of certain exponential moments.

\begin{theorem}\label{thrm:moments exp}
Fix $\alpha<2$, $p\in(0,1]$, and $\theta>0$. Let $\mu=TS_\alpha^p(R,b)$.\\
1. If $\alpha\in(0,2)$ then
$$
\int_{\mathbb{R}^d}e^{\theta|x|^p}\mu(\rd x)<\infty \Longleftrightarrow R(\{|x|>\theta^{-1/p}\})=0.
$$
2. If $\alpha<0$ then $\int_{\mathbb R^d}e^{\theta|x|^p}\mu(\rd x)<\infty$ if and only if
$$
R(\{|x| \ge \theta^{-1/p}\})=0\ \mbox{and}\ \int_{0<|x|^{-p}-\theta<1}(|x|^{-p}-\theta)^{\alpha/p} R(\rd x)<\infty.
$$
3. If $\alpha = 0$ then $\int_{\mathbb R^d}e^{\theta|x|^p}\mu(\rd x)<\infty$ if and only if
$$
R(\{|x| \ge \theta^{-1/p}\})=0\ \mbox{and}\
\int_{0<|x|^{-p}-\theta<1}\left|\log(|x|^{-p}-\theta)\right|R(\rd x)<\infty.
$$
\end{theorem}

This implies that unless $R=0$ it is impossible to have
$
\int_{\mathbb R^d} e^{\theta|x|^p} \mu(\rd x)<\infty\ \mbox{for\ all}\ \theta>0.
$
Note that in Parts 2 and 3 we have the condition, $R(\{|x| \ge \theta^{-1/p}\})=0$, whereas in Part 1 we have
a similar condition, but with strict inequality. Note also that the set $\{0<|x|^{-p}-\theta<1\} =
\{(1+\theta)^{-1/p}<|x|<\theta^{-1/p}\}$. The latter form may be somewhat more appealing, but it loses
emphasis on why the integrals may diverge.

\begin{proof}
The proof of Part 1 is similar to the proof of Proposition 2.7 in \cite{Rosinski:2007}. Now fix
$\alpha\le0$. By Corollary 25.8 in \cite{Sato:1999} the finiteness of
$\int_{\mathbb{R}^d}e^{\theta|x|^p}\mu(\rd x)$ is equivalent to the finiteness of
$\int_{|x|>1}e^{\theta|x|^p}M(\rd x)$. We have
\begin{eqnarray*}
\int_{|x|>1}e^{\theta|x|^p}M(\rd x) &=&
\int_{\mathbb{R}^d}\int^\infty_{|x|^{-1}}e^{(\theta|x|^p-1)t^p}t^{-\alpha-1}\rd tR(\rd x)\\
&\ge& \int_{|x|^p\ge\theta^{-1}}\int_{\theta^{1/p}}^\infty t^{-\alpha-1}\rd t R(\rd x).
\end{eqnarray*}
This shows the necessity of $R(\{|x| \ge \theta^{-1/p}\})=0$ in both Parts 2 and 3. We will
henceforth assume that this property holds both when showing necessity and sufficiency. We have
\begin{eqnarray*}
\int_{|x|>1}e^{\theta|x|^p}M(\rd x) &=&
\int_{|x|<\theta^{-1/p}}\int_{|x|^{-1}}^{\infty}e^{(\theta|x|^p-1)t^p}t^{-1-\alpha}\rd t R(\rd x)\\
&=& p^{-1}\int_{0<|x|^{-p}-\theta}(1-\theta|x|^p)^{\alpha/p}\int_{|x|^{-p}-\theta}^\infty
e^{-u}u^{-1-\alpha/p} \rd u R(\rd x).
\end{eqnarray*}
This can be divided into two parts
\begin{eqnarray*}
&& p^{-1} \int_{1\le|x|^{-p}-\theta}(|x|^{-p}-\theta)^{-|\alpha|/p}|x|^{-|\alpha|}\int_{|x|^{-p}-\theta}^\infty
e^{-u}u^{-1+|\alpha|/p} \rd uR(\rd x)\\
&& \quad +p^{-1} \int_{0<|x|^{-p}-\theta<1} (|x|^{-p}-\theta)^{-|\alpha|/p}|x|^{-|\alpha|} \int_{|x|^{-p}-\theta}^\infty
e^{-u}u^{-1+|\alpha|/p} \rd uR(\rd x)\\
&&=: p^{-1}(I_1 + I_2).
\end{eqnarray*}
Let $C_\theta:=\sup_{u>1}e^{-u}u^{-1+|\alpha|/p}(u+\theta)^{(2+|\alpha|)/p+1}$. We have,
\begin{eqnarray*}
I_1&\le& \int_{1\le|x|^{-p}-\theta}|x|^{-|\alpha|} \int_{|x|^{-p}-\theta}^\infty e^{-u}u^{-1+|\alpha|/p}\rd u R(\rd x)\\
&\le& C_\theta \int_{1\le|x|^{-p}-\theta}|x|^{-|\alpha|} \int_{|x|^{-p}-\theta}^\infty (u+\theta)^{-(2+|\alpha|)/p-1}\rd uR(\rd x)\\
&=& C_\theta \frac{p}{2-\alpha}\int_{|x|\le (1+\theta)^{-1/p}}|x|^2R(\rd x)<\infty.
\end{eqnarray*}
Thus finiteness is determined by $I_2$.

If $\alpha<0$ and $0<|x|^{-p}-\theta<1$, we have
\begin{eqnarray*}
\int_1^\infty e^{-u}u^{-1+|\alpha|/p}\rd u \le \int_{|x|^{-p}-\theta}^\infty
e^{-u}u^{-1+|\alpha|/p}\rd u \le \Gamma(|\alpha|/p)
\end{eqnarray*}
and
$$
\theta^{|\alpha|/p} \le |x|^{-|\alpha|} \le (1+\theta)^{|\alpha|/p}.
$$
Thus, when $\alpha<0$, $I_2$ is finite if and only if
\begin{eqnarray*}
\int_{0<|x|^{-p}-\theta<1}(|x|^{-p}-\theta)^{-|\alpha|/p} R(\rd x)<\infty
\end{eqnarray*}

If $\alpha=0$ then for $0<|x|^{-p}-\theta<1$, we have
\begin{eqnarray*}
\int_{|x|^{-p}-\theta}^\infty e^{-u}u^{-1}\rd u &=& \int_1^\infty
e^{-u}u^{-1}\rd u + \int_{|x|^{-p}-\theta}^1e^{-u}u^{-1}\rd u,
\end{eqnarray*}
where the first integral is finite.  For the second, we have
\begin{eqnarray*}
\int_{|x|^{-p}-\theta}^1e^{-u}u^{-1}\rd u &\le& \int_{|x|^{-p}-\theta}^1u^{-1}\rd u = -\log(|x|^{-p}-\theta),
\end{eqnarray*}
and
\begin{eqnarray*}
\int_{|x|^{-p}-\theta}^1 e^{-u}u^{-1}\rd u &\ge& e^{-1}\int_{|x|^{-p}-\theta}^1 u^{-1}\rd u = - e^{-1}\log(|x|^{-p}-\theta).
\end{eqnarray*}
Thus, when $\alpha=0$, the finiteness of $I_2$ is equivalent to the finiteness of
\begin{eqnarray*}
-\int_{0<|x|^{-p}-\theta<1}\log(|x|^{-p}-\theta)R(\rd x) = \int_{0<|x|^{-p}-\theta<1}\left|\log(|x|^{-p}-\theta)\right|R(\rd x).
\end{eqnarray*}
This completes the proof.
\end{proof}

\begin{theorem}
Fix $\alpha<2$, $p>0$, and let $\mu=TS_\alpha^p(R,b)$.\\
1. If $q\in(0,1]$ with $q<p$ then for any $\theta>0$ 
\begin{eqnarray}\label{eq: finite q exp moment}
\int_{\mathbb{R}^d}e^{\theta|x|^q}\mu(\rd x)<\infty
\end{eqnarray}
whenever
\begin{eqnarray}\label{eq: suff cond for exp moment}
\int_{|x|>1}e^{A_{p,q}(\theta|x|^q)^{p/(p-q)}}|x|^{-\alpha q/(p-q)}R(\rd x)<\infty,
\end{eqnarray}
where $A_{p,q}=\left(q/p\right)^{q/(p-q)}\left(1-q/p\right)$.\\
2. If $R\ne0$ then $\int_{\mathbb{R}^d}e^{\theta|x|\log|x|}\mu(\rd x)=\infty$ for every $\theta>0$.
\end{theorem}

For the case where $\alpha\in(0,2)$, $p=2$, and $q=1$, a necessary and sufficient condition for \eqref{eq: finite q exp moment} is given in \cite{Bianchi:Rachev:Kim:Fabozzi:2011}. Their method of proof is easily extended to the case when $\alpha<2$ and $p=2q$. In this case, the necessary and sufficient condition for \eqref{eq: finite q exp moment} is
\begin{eqnarray}
\int_{|x|>1} e^{\theta^2|x|^{2q}/4}|x|^{-q-\alpha}R(\rd x)<\infty.
\end{eqnarray}

\begin{proof}
We begin with Part 1. Fix $c=\left(\frac{2p}{\theta q}\right)^{1/p}$. By Corollary 25.8 in \cite{Sato:1999},
the problem is equivalent to the finiteness of
\begin{eqnarray*}
\int_{|x|>1}e^{\theta|x|^q}M(\rd x) &=& \int_{|x|\le c}\int_{|x|^{-1}}^\infty e^{\theta|x|^q t^{q}-t^p}t^{-1-\alpha}\rd tR(\rd x)\\
&& +\int_{|x|> c}\int_{|x|^{-1}}^{(\theta|x|^q)^{1/(p-q)}} e^{\theta|x|^q t^{q}-t^p}t^{-1-\alpha}\rd tR(\rd x)\\
&& + \int_{|x|> c}\int_{(\theta|x|^q)^{1/(p-q)}} ^\infty e^{\theta|x|^q t^q-t^p}t^{-1-\alpha}\rd tR(\rd x)\\
&=:& I_1+I_2 + I_3.
\end{eqnarray*}
For the first integral, we have
\begin{eqnarray*}
I_1 &=& \int_{|x|\le c }\int_{|x|^{-1}}^\infty e^{\theta|x|^q t^{q}-t^p}t^{1-\alpha}t^{-2}\rd tR(\rd x)\\
&\le& \int_{|x|\le c}|x|^2R(\rd x)\int_{c^{-1}}^\infty e^{\theta c^qt^q-t^p}t^{1-\alpha}\rd t<\infty.
\end{eqnarray*}
For the third integral, by the substitution $u=t^{p-q}/(\theta|x|^q)$, we have
\begin{eqnarray*}
I_3 &=& \frac{1}{(p-q)}\int_{|x|> c }(\theta|x|^q)^{-\alpha/(p-q)}\times\\
&&\qquad \times\int_1^\infty e^{-\left(1-1/u\right)(u\theta|x|^q)^{p/(p-q)}} u^{-1-\alpha/(p-q)} \rd uR(\rd x)\\
&\le& \frac{1}{(p-q)}\int_{|x|> c}(\theta|x|^q)^{-\alpha/(p-q)}R(\rd x)\times\\
&&\qquad \times\int_1^\infty e^{-\left(1-1/u\right)(u \theta c^q)^{p/(p-q)}} u^{-1-\alpha/(p-q)} \rd u.
\end{eqnarray*}
Clearly this is finite for $\alpha\in[0,2)$. We will show that it is, in fact, always finite when $I_2<\infty$. To see this, observe that, after the substitution $u=t^{p-q}/(\theta|x|^q)$, we have
\begin{eqnarray}
I_2 &=& \frac{1}{(p-q)}\int_{|x|> c}(\theta|x|^q)^{-\alpha/(p-q)}\times\nonumber\\
&&\qquad \times \int_{|x|^{-p}/\theta}^1 e^{\left(1/u-1\right)(u\theta|x|^q)^{p/(p-q)}} u^{-1-\alpha/(p-q)} \rd u R(\rd x)\label{eq: I2 repar}\\
&\ge& \frac{1}{(p-q)}\int_{|x|> c}(\theta|x|^q)^{-\alpha/(p-q)}R(\rd x)\int_{c^{-p}/\theta}^1 u^{-1-\alpha/(p-q)} \rd u.\nonumber
\end{eqnarray}
Thus everything is determined by $I_2$.

Note that, as a function of $u$, $\left(1/u-1\right)(u\theta|x|^q)^{p/(p-q)}$ is strictly increasing until $u=q/p$, where it attains a maximum and is then decreasing. Thus
\begin{eqnarray}
&& \int_{|x|^{-p}/\theta}^{q/(2p)} e^{\left(1/u-1\right)(u\theta|x|^q)^{p/(p-q)}} u^{-1-\alpha/(p-q)} \rd u
\nonumber \\
&&\qquad  \le e^{\left(2p/q-1\right)(\theta|x|^q q/(2p))^{p/(p-q)}}
(|x|^p\theta)^{0\vee[1+\alpha/(p-q)]}\label{eq: first half}
\end{eqnarray}
and for some constant $C>0$
\begin{eqnarray}\label{eq: second half}
\int_{q/(2p)}^1  e^{\left(1/u-1\right)(u\theta|x|^q)^{p/(p-q)}} u^{-1-\alpha/(p-q)} \rd u \le C e^{\left(p/q-1\right)[q p^{-1}\theta|x|^q]^{p/(p-q)}}.
\end{eqnarray}
Note that $\left(p/q-1\right)(q/p)^{p/(p-q)}=A_{p,q}$. Now observing that the right side in \eqref{eq:
second half} goes to infinity faster than the right side in \eqref{eq: first half}, and combining this with
\eqref{eq: I2 repar} gives Part 1.

Now to show Part 2. For any $h>0$, let $T_h=\{|x|>h\}$. Assume that $R\ne0$. Since $R(\{0\})=0$ there exists
an
$\epsilon>0$ such that $R(T_\epsilon)>0$. Thus for any $h>0$
\begin{eqnarray*}
M(T_h) &=& \int_{\mathbb R^d}\int_{h|x|^{-1}}^\infty e^{-t^p}t^{-1-\alpha}\rd t R(\rd x)\\
&\ge& \int_{|x|>\epsilon}\int_{h\epsilon^{-1}}^\infty e^{-t^p}t^{-1-\alpha}\rd t R(\rd x)
= R(T_\epsilon)\int_{h\epsilon^{-1}}^\infty e^{-t^p}t^{-1-\alpha}\rd t>0.
\end{eqnarray*}
From here the result follows by Theorem 26.1 in \cite{Sato:1999}.
\end{proof}

\section{Regular Variation}\label{sec: Reg Var for TS}

In this section we give necessary and sufficient conditions for tempered stable distributions to have regularly varying tails. To simplify the notation, we adopt the following convention. For $c\in\mathbb R$ and real-valued functions $f,g$ with $g$ strictly positive in some neighborhood of infinity we write $f(t)\sim cg(t)$ as $t\rightarrow\infty$ to mean
$$
\lim_{t\rightarrow\infty}\frac{f(t)}{g(t)}=c.
$$
We now recall what it means for a measure to have regularly varying tails.

\begin{definition}
Fix $\varrho\ge0$. Let $R$ be a Borel measure on $\mathbb R^d$ such that for some $T>0$
\begin{eqnarray}
R(\{x \in \mathbb R^d : |x| > T\}) < \infty
\end{eqnarray}
and for all $s > 0$
\begin{eqnarray}
R(\{x \in\mathbb R^d : |x| > s \}) > 0.
\end{eqnarray}
We say that $R$ has \textbf{regularly varying tails with index $\varrho$} if there exists a finite Borel measure $\sigma\ne 0$ on $\mathbb S^{d-1}$ such that for all $D\in \mathfrak B(\mathbb S^{d-1})$ with $\sigma(\partial D)=0$
\begin{eqnarray}
\lim_{r\downarrow0}\frac{R\left(|x|>rt:\frac{x}{|x|}\in D\right)}{R\left(|x|>r\right)} = t^{-\varrho}\frac{\sigma(D)}{\sigma(\mathbb S^{d-1})}.
\end{eqnarray}
When this holds we write $R \in RV_{-\varrho}(\sigma)$.
\end{definition}

Clearly a measure $R\in RV_{-\varrho}(\sigma)$ if and only if there exists a slowly varying function $\ell$ such that for all $D\in\mathfrak B(\mathbb S^{d-1})$ with $\sigma(\partial D)=0$
\begin{eqnarray}\label{eq: alt defn of reg var meas}
R(|x|>t, x/|x|\in D) \sim \sigma(D) t^{-\varrho}\ell(t) \ \mbox{as\ } t\rightarrow\infty.
\end{eqnarray}
It is well-known (see e.g.\ \cite{Basrak:Davis:Mikosch:2002}) that if $R\in RV_{-\varrho}(\sigma)$ then
\begin{eqnarray}\label{eq: moments of RV measure}
\int_{|x|\ge T}|x|^\gamma R(\rd x)\left\{
\begin{array}{lr}
<\infty & \mathrm{if}\ \gamma<\varrho\\
=\infty & \mathrm{if}\ \gamma>\varrho
\end{array}\right..
\end{eqnarray}

Let $\mu=\ts(R,b)$. If $\alpha\in(0,2)$ then Theorem \ref{thrm:moments} implies that $\int_{\mathbb R^d}|x|^\varrho\mu(\rd x)<\infty$ for all $\varrho\in[0,\alpha)$, and hence, by \eqref{eq: moments of RV measure} $\mu$ cannot have regularly varying tails with index $\varrho <\alpha$. However, other tail indices are possible. We will now categorize when $\mu$ has regularly varying tails. 

\begin{theorem}\label{thrm: eq rv R and M}
Fix $\alpha<2$, $p>0$. Let $\mu=TS^p_\alpha(R,b)$ and let $M$ be the L\'evy measure of $\mu$. If $\varrho > \alpha\vee0$ then
\begin{eqnarray}
\mu\in RV^\infty_{-\varrho}(\sigma) \Longleftrightarrow M\in RV^\infty_{-\varrho}(\sigma) \Longleftrightarrow R\in RV^\infty_{-\varrho}(\sigma).
\end{eqnarray}
Moreover, if $M\in RV^\infty_{-\varrho}(\sigma)$ then for all $D\in\mathfrak B(\mathbb S^{d-1})$ with $\sigma(\partial D)=0$ and $\sigma(D)>0$
$$
\lim_{r\rightarrow\infty}\frac{R\left(|x|>r, x/|x|\in D\right)}{M\left(|x|>r, x/|x|\in D\right)} = \frac{p}{\Gamma\left(\frac{\varrho-\alpha}{p}\right)}.
$$
\end{theorem}

Before proving the theorem let us state a useful corollary. Recall that for $\gamma\in(0,2)$ a probability
measure $\mu$ is in the domain of attraction of a $\gamma$-stable distribution with spectral measure
$\sigma\ne0$ if and only if $\mu\in RV_{-\gamma}(\sigma)$. See e.g.\ \cite{Rvaceva:1962} or
\cite{Meerschaert:Scheffler:2001} although they make the additional assumption that the limiting stable
distribution is full.

\begin{corollary}\label{cor: doa for ts}
Fix $\alpha<2$, $p>0$, and let $\mu=\ts(R,b)$. If $\sigma\ne0$ is a finite Borel measure on $\mathbb S^{d-1}$ and $\gamma\in(0\vee\alpha,2)$ then $\mu$ is in the domain of attraction of a $\gamma$-stable with spectral measure $\sigma$ if and only if $R\in RV_{-\gamma}(\sigma)$.
\end{corollary}

In Theorem \ref{thrm: eq rv R and M}, the relationship between the regular variation of $\mu$ and $M$ is well
know, see for example \cite{Hult:Lindskog:2006}. A proof of the fact that $R\in RV^\infty_{-\varrho}(\sigma)$
implies that $M\in RV^\infty_{-\varrho}(\sigma)$ can be accomplished using standard tools. 
However the other direction requires heavier machinery. For brevity, we use the same approach for both
directions.

Let $k:(0,\infty)\mapsto\mathbb R$ be a Borel function. The \textbf{Mellin transform} of $k$ is defined by
\begin{eqnarray}
\hat k(z) = \int_0^\infty u^{z-1}k(1/u)\rd u
\end{eqnarray}
for all $z\in\mathbb C$ for which the integral converges. We will need the following result, which combines Theorems 4.4.2 and 4.9.1 in \cite{Bingham:Goldie:Teugels:1987}.

\begin{lemma}\label{lemma: Tauberian thrm for Stieltjes}
Let $-\infty<\gamma<\rho<\tau<\infty$, $c\in\mathbb R$, and let $\ell$ be a slowly varying function. Assume that $k$ is a continuous and non-negative function on $(0,\infty)$ such that
\begin{eqnarray}\label{eq: sum cond}
\sum_{-\infty<n<\infty} \max\{e^{-\gamma n},e^{-\tau n}\} \sup_{e^n\le x\le e^{n+1}}k(x)<\infty
\end{eqnarray}
and
\begin{eqnarray}\label{eq: k hat not zero}
\hat k(z)\ne0 \ \ \mathrm{when} \ \Re z=\rho.
\end{eqnarray}
Let $U$ be a monotone, right continuous function on $(0,\infty)$ with
\begin{eqnarray}
\limsup_{r\downarrow0}\frac{|U(r)|}{r^\gamma}<\infty.
\end{eqnarray}
Then
\begin{eqnarray}
\int_{(0,\infty)} k(x/t)\rd U(t) \sim c\rho \hat k(\rho)x^\rho\ell(x) \ \ \mathrm{as}\ x\rightarrow\infty
\end{eqnarray}
if and only if
\begin{eqnarray}
U(x) \sim cx^\rho\ell(x) \ \ \mathrm{as}\ x\rightarrow\infty.
\end{eqnarray}
\end{lemma}

Let $\mu=\ts(R,b)$, let $M$ be the L\'evy measure of $\mu$, and assume that $\sigma\ne0$ is a finite Borel measure on $\mathbb S^{d-1}$.  For all $D\in \mathfrak B(\mathbb S^{d-1})$ with $\sigma(\partial D)=0$ define for $r>0$
$$
M_D(r) = M(|x|>r, x/|x|\in D)
$$
and
$$
R_D(r) = R(|x|>r, x/|x|\in D).
$$
Note that for any integrable function $f:\mathbb R\rightarrow \mathbb R$
\begin{eqnarray}\label{eq: integ RD}
\int_{x/|x|\in D} f(|x|)R(\rd x) = -\int_{(0,\infty)} f(x)\rd R_D(x).
\end{eqnarray}

\begin{lemma}\label{lemma: prep for rv of ts}
If $\varrho>\alpha\vee0$ and $\ell\in RV^\infty_0$ then
\begin{eqnarray*}
M_D(r)\sim \sigma(D)p^{-1}\Gamma\left(\frac{\varrho-\alpha}{p}\right)r^{-\varrho}\ell(r) \ \mbox{as} \  r\rightarrow\infty
\end{eqnarray*}
if and only if
\begin{eqnarray*}
R_D(r)\sim \sigma(D) r^{-\varrho}\ell(r) \ \mbox{as} \  r\rightarrow\infty.
\end{eqnarray*}
\end{lemma}

\begin{proof}
For simplicity, let $\beta=\alpha\vee0$. Note that by \eqref{eq:levy m} and \eqref{eq: integ RD}
\begin{eqnarray*}
M_D(r) &=& \int_{x/|x|\in D}\int_{r|x|^{-1}}^\infty t^{-1-\alpha}e^{-t^p}\rd tR(\rd x)\\
&=& -\int_{(0,\infty)} \int_{r/x}^\infty t^{-1-\alpha}e^{-t^p}\rd t \rd R_D(x)\\
&=& -\int_{(0,\infty)} k(r/x) R_D(\rd x),
\end{eqnarray*}
where
\begin{eqnarray*}
k(s) = \int_s^\infty t^{-\alpha-1}e^{-t^p}\rd t = p^{-1}\int_{s^p}^\infty t^{-\alpha/p-1}e^{-t}\rd t.
\end{eqnarray*}
For $\Re z<-\beta$
\begin{eqnarray*}
\hat k(z) &=& \int_0^\infty u^{z-1}k(1/u)\rd u = \int_0^\infty u^{z-1}\int_{1/u}^\infty t^{-1-\alpha}e^{-t^p} \rd t\rd u\\
&=& \int_0^\infty u^{z+\alpha-1}\int_1^\infty t^{-1-\alpha}e^{-(t/u)^p}\rd t\rd u\\
&=& \int_0^\infty u^{-z-\alpha-1}e^{-u^p}\rd u\int_1^\infty t^{z-1}\rd t = -\frac{1}{pz}\Gamma\left(\frac{-z-\alpha}{p}\right).
\end{eqnarray*}
Thus, since $-\varrho<-\beta$
$$
k(-\varrho) = \frac{1}{p\varrho}\Gamma\left(\frac{\varrho-\alpha}{p}\right).
$$
From here, the result will follow by ma \ref{lemma: Tauberian thrm for Stieltjes}. We just need to verify that the assumptions hold.

It is easy to see that $k$ is a continuous, non-negative function on $(0,\infty)$ and that $\hat k(z)$ has no zeros. Fix $\tau\in(-\varrho,-\beta)$, $\gamma<-(\varrho\vee2)$, and let $C=\sup_{t\ge1}t^{-\alpha/p-1}e^{-t/2}$. Note that $\gamma<\tau<0$. We have
\begin{eqnarray*}
&& p\sum_{n=0}^\infty \max\{e^{-\gamma n},e^{-\tau n}\} \sup_{e^n\le x\le e^{n+1}}k(x)\\
&&\qquad\qquad = \sum_{n=0}^\infty e^{|\gamma| n} \int_{e^{np}}^\infty t^{-\alpha/p-1}e^{-t/2}e^{-t/2}\rd t\\
&&\qquad\qquad \le C\sum_{n=0}^\infty e^{|\gamma| n} \int_{e^{np}}^\infty e^{-t/2}\rd t = 2C\sum_{n=0}^\infty e^{|\gamma| n} e^{-e^{np}/2}<\infty
\end{eqnarray*}
and
\begin{eqnarray*}
&&p\sum_{-\infty<n\le -1} \max\{e^{-\gamma n},e^{-\tau n}\} \sup_{e^n\le x\le e^{n+1}}k(x)\\
&&\qquad\qquad = \sum_{n=1}^\infty e^{-|\tau| n} \int_{e^{-np}}^\infty t^{-\alpha/p-1}e^{-t}\rd t.
\end{eqnarray*}
When $\alpha<0$ this is bounded by
\begin{eqnarray*}
\int_0^\infty t^{-\alpha/p-1}e^{-t}\rd t\sum_{n=1}^\infty e^{-|\tau| n} <\infty.
\end{eqnarray*}
When $\alpha=0$ it is bounded by
\begin{eqnarray*}
\sum_{n=1}^\infty e^{-|\tau| n} \left(\int_{e^{-np}}^1 t^{-1}\rd t+\int_1^\infty e^{-t}\rd t\right) =\sum_{n=1}^\infty e^{-|\tau| n}\left(np +e^{-1}\right)<\infty.
\end{eqnarray*}
When $\alpha\in(0,2)$ it is bounded by
\begin{eqnarray*}
\sum_{n=1}^\infty e^{-|\tau| n} \int_{e^{-np}}^\infty t^{-\alpha/p-1}\rd t = \frac{p}{\alpha}\sum_{n=1}^\infty e^{-|\tau| n+\alpha n} = \frac{p}{\alpha}\sum_{n=1}^\infty e^{-(|\tau|-\alpha) n}<\infty.
\end{eqnarray*}

Recall that $\gamma<-2$. Note that $-R_D(r)$ is a right continuous, monotonely increasing function on $(0,\infty)$ with
\begin{eqnarray}
\limsup_{r\downarrow0}\frac{|-R_D(r)|}{r^\gamma} &\le& \limsup_{r\downarrow0}r^2\int_{|x|>r} R(\rd x)\nonumber\\
&\le& \limsup_{r\downarrow0}\int_{1>|x|>r} |x|^2 R(\rd r) + \limsup_{r\downarrow0}r^2R\left(|x|\ge1\right)\nonumber\\
&\le& \int_{|x|<1} |x|^2 R(\rd r)<\infty.\nonumber
\end{eqnarray}
This completes the proof of Lemma \ref{lemma: prep for rv of ts}.
\end{proof}

The proof of Theorem \ref{thrm: eq rv R and M} follows immediately from Lemma \ref{lemma: prep for rv of ts} and \eqref{eq: alt defn of reg var meas}.

\section*{Acknowledgements}

Most of the research for this paper was done while the author was a PhD student working with Professor Gennady Samorodnitsky.  Professor Samorodnitsky's comments and support are gratefully acknowledged.

\bibliographystyle{plain}
\bibliography{Grabchak}

\end{document}